\documentclass{amsart}
\usepackage{amsmath, amssymb, mathrsfs, verbatim, multirow}
\usepackage[all]{xy}

\newtheorem{Teo}{Theorem}[section]
\newtheorem{Prop}[Teo]{Proposition}
\newtheorem{Lema}[Teo]{Lemma}
\newtheorem{Cor}[Teo]{Corollary}

\theoremstyle{definition}
\newtheorem{Def}[Teo]{Definition}

\newtheorem{Obs}[Teo]{Remark}

\newcommand{\N}{\mathbb{N}}

\newcommand{\Llr}{\Longleftrightarrow}
\newcommand{\lra}{\longrightarrow}

\newcommand{\VR}{\mathcal{O}}
\newcommand{\PI}{\mathfrak{p}}
\newcommand{\MI}{\mathfrak{m}}
\newcommand{\QI}{\mathfrak{q}}

\newcommand{\hei}{\mbox{\rm ht}}

\newcommand{\SU}{\mbox{\rm supp}}
\newcommand{\ass}{\mbox{\rm Ass}}
\newcommand{\rad}{\mbox{\rm Nil}}
\newcommand{\ann}{\mbox{\rm ann}}

\newcommand{\rk}{\mbox{\rm rk}}
\newcommand{\Sp}{\mbox{\rm Spec}}

\begin{document}
\title[Local uniformization for non-domains]{Reduction of local uniformization to the case of rank one valuations for rings with zero divisors}

\author{Josnei Novacoski}
\address{Josnei Novacoski\newline\indent CAPES Foundation  \newline \indent Ministry of Education of Brazil \newline \indent Bras\'ilia/DF 70040-020 \newline \indent Brazil}
\email{jan328@mail.usask.ca}
\author{Mark Spivakovsky}

\address{Mark Spivakovsky \newline\indent Institut de Math\'ematiques de Toulouse and CNRS\newline\indent Universit\'e Paul Sabatier \newline\indent  118 route de Narbonne \newline\indent  F-31062 Toulouse cedex 9 \newline\indent  France}
\email{mark.spivakovsky@math.univ-toulouse.fr}

\thanks{During the realization of this project the first author was supported by a grant from the program ``Ci\^encia sem Fronteiras" from the Brazilian government.}

\begin{abstract}
This is a continuation of a previous paper by the same authors. In the former paper, it was proved that in order to obtain local uniformization for valuations centered on local domains, it is enough to prove it for rank one valuations. In this paper, we extend this result to the case of valuations centered on rings which are not necessarily integral domains and may even contain nilpotents.
\end{abstract}

\maketitle
\section{Introduction}
For an algebraic variety $X$ over a field $k$, the problem of resolution of singularities is whether there exists a proper birational morphism $X'\lra X$ such that $X'$ is regular. The problem of local uniformization can be seen as the local version of resolution of singularities for an algebraic variety.  For a valuation $\nu$ of $k(X)$ having a center on $X$, the local uniformization problem asks whether there exists a proper birational morphism $X'\lra X$ such that the center of $\nu$ on $X'$ is regular. This problem was introduced by Zariski in the 1940's as an important step to prove resolution of singularities. Zariski's approach consists in proving first that every valuation having a center on the given algebraic variety admits local uniformization. Then one has to glue these local solutions to obtain a global resolution of all singularities.

Zariski succeeded in proving local uniformization for valuations centered on algebraic varieties over a field of characteristic zero (see \cite{Zar_2}). He used this to prove resolution of singularities for algebraic surfaces and threefolds  over a field of characteristic zero (see \cite{Zar_4}). Abhyankar proved (see \cite{Ab_1}) that local uniformization can be obtained for valuations centered on algebraic surfaces in any characteristic and used this fact to prove resolution of singularities for surfaces (see \cite{Ab_2} and \cite{Ab_3}). He also proved local uniformization and resolution of singularities for threefolds over fields of characteristic other than 2, 3 and 5 (see \cite{Ab_4}). Very recently, Cossart and Piltant proved resolution of singularities (and, in particular, local uniformization) for threefolds over any field of positive characteristic, as well as in the arithmetic case (see \cite{Cos_2} and \cite{Cos_3}). They proved it using the approach of Zariski. However, the problem of local uniformization remains open for valuations centered on algebraic varieties of dimension greater than three over fields of positive characteristic.

Since local uniformization is a local problem, we can work with local rings instead of algebraic varieties. A valuation $\nu$ centered on a local integral domain $R$ is said to admit local unifomization if there exists a local local ring $R^{(1)}$ dominated by $\VR_\nu$ and dominating $R$ such that $R^{(1)}$ is regular. Let $\mathcal{N}$ be the category of all noetherian local domains and $\mathcal{M}\subseteq\mathcal{N}$ be a subcategory of $\mathcal{N}$ which is closed under taking homomorphic images and localizing any finitely generated birational extension at a prime ideal. We want to know for which subcategories $\mathcal M$ with these properties, all valuations centered on objects of $\mathcal M$ admit local uniformization. In Section 7.8 of \cite{EGAIV}, Grothendieck proved that any category of schemes, closed under passing to subschemes and finite radical extensions, in which resolution of singularities holds, is a subcategory of quasi-excellent schemes (it is known that the category of quasi-excellent schemes is closed under all the operations mentioned above). He conjectured (see Remark 7.9.6 of \cite{EGAIV}) that resolution of singularities holds in this most general possible context: that of quasi-excellent schemes. Translated into our local situation, this conjecture says that the subcategory of
$\mathcal{N}$ which optimizes local uniformization is the category of all quasi-excellent local rings. This subcategory has the properties above. For a discussion on quasi-excellent and excellent local rings see Section 7.8 of \cite{EGAIV}. However, this conjecture is widely open.

In most of the successful cases, including those mentioned above, local uniformization was first proved for rank one valuations. Then the general case was reduce to this a priori weaker one. In \cite{Nov1}, we prove that this reduction works under very general assumptions. Namely, we consider a subcategory $\mathcal M$ of the category of all noetherian local integral domains, closed under taking homomorphic images and localizing any finitely generated birational extension at a prime ideal. The main result of \cite{Nov1} is that if every rank one valuation centered on an object of $\mathcal M$ admits local uniformization, then all the valuations centered on objects of $\mathcal M$ admit local uniformization. The main goal of this paper is to extend this result to rings which are not necessarily integral domains and, in particular, may contain nilpotent elements. The importance of non-integral and non-reduced schemes in modern algebraic geometry is well known. Even if one were only interested in reduced schemes to start with, one is led to consider non-reduced ones as they are produced by natural constructions, for example, in deformation theory. Therefore, it appears desirable to study the problem of local uniformization for such schemes and, in particular, to extend our earlier results on reducing the problem to the rank one case to this more general context.    

If $R$ is not reduced we cannot expect, in general, to make $R^{(1)}$ be regular by blowings up. The natural extension to this case is to require $\left(R^{(1)}\right)_{\rm red}$ to be regular and $I_{(1)}^n/I_{(1)}^{n+1}$ to be an $\left(R^{(1)}\right)_{\rm red}$-free module for every $n\in\N$ (here $I_{(1)}$ denotes the nilradical of $R^{(1)}$). For more precise definitions see Section
\ref{Preliminaires}. Let $\mathcal{N}$ be the category of all noetherian local rings and $\mathcal{M}\subseteq\mathcal{N}$ be a subcategory of $\mathcal{N}$ which is closed under taking homomorphic images and localizing any finitely generated birational extension at a prime ideal. Our main result is the following:

\begin{Teo}\label{mainthm}
Assume that for every noetherian local ring $R$ in $Ob(\mathcal{M})$, every rank one valuation centered on $R$ admits local uniformization. Then all the valuations centered on objects of $\mathcal{M}$ admit local uniformization.
\end{Teo}

The proof of Theorem \ref{mainthm} consists of three main steps. The first step is to prove that for every local ring $R$ and every valuation $\nu$ centered on $R$, there exists a local blowing up (see Definition \ref{defblup}) $R\lra R^{(1)}$ such that $R^{(1)}$ has only one associated prime ideal. Then we consider a decomposition $\nu=\nu_1\circ\nu_2$ of $\nu$ such that $\rk(\nu_1)<\rk(\nu)$ and $\rk(\nu_2)<\rk(\nu)$. Using induction, we can assume that both $\nu_1$ and $\nu_2$ admit local uniformization. The second main step consists in using this to prove that there exists a local blowing up $R^{(1)}\lra R^{(2)}$ such that $\left(R^{(2)}\right)_{\rm red}$ is regular. The third and final step is to prove that there exists a further local blowing up $R^{(2)}\lra R^{(3)}$ such that $\left(R^{(3)}\right)_{\rm red}$ is regular and $I_{(3)}^n/I_{(3)}^{n+1}$ is an $\left(R^{(3)}\right)_{\rm red}$-free module for every $n\in\N$ (here $I_{(3)}$ denotes the nilradical of $R^{(3)}$).

This paper is divided as follows. In Section \ref{Preliminaires} we present the basic definitions and results that will be used in the sequel. Sections \ref{assocprimeideals}, \ref{sectiomakerpireg} and \ref{sectiomakeinin1free} are dedicated to prove the results related to the first, second and third steps, respectively. In the last section we present a proof of our main theorem.

\section{Preliminaries}\label{Preliminaires}

Let $R$ be a noetherian commutative ring with unity and $\Gamma$ an ordered abelian group. Set $\Gamma_\infty:=\Gamma\cup\{\infty\}$ and extend the addition and order from $\Gamma$ to $\Gamma_\infty$ as usual.
\begin{Def}
A valuation $\nu$ on $R$ is a mapping $\nu:R\lra \Gamma_\infty$ with the following properties:
\begin{itemize}
\item[(V1)] $\nu(ab)=\nu(a)+\nu(b)$ for every $a,b\in R$;
\item[(V2)] $\nu(a+b)\geq\min \{\nu(a),\nu(b)\}$ for every $a,b\in R$;
\item[(V3)] $\nu(1)=0$ and $\nu(0)=\infty$;
\item[(V4)] The \textbf{support of $\nu$}, which is defined by $\SU(\nu):=\{a\in R\mid \nu(a)=\infty\}$, is a minimal prime ideal of $R$.
\end{itemize}
\end{Def}

Take a multiplicative system $S$ of $R$ such that $\SU(\nu)\subseteq R\setminus S$. Then the extension (which we call again $\nu$) of $\nu$ to $R_S$ given by $\nu(a/s):=\nu(a)-\nu(s)$ is again a valuation. Indeed, the three first axioms are easily checked. The minimality of $\SU(\nu)$ as a prime ideal of $R_S$ follows from the fact that the prime ideals of $R_S$ are in a bijective correspondence to the prime ideals of $R$ contained in $R\setminus S$. From now on, we will freely make such extensions of $\nu$ to
$R_S$ without mentioning it explicitly.

A valuation $\nu$ on $R$ is said to have a center if $\nu(a)\geq 0$ for every $a\in R$. In this case, the \textbf{center} of $\nu$ on
$R$ is defined by $\mathfrak C_{\nu}(R):=\{a\in R\mid\nu(a)>0\}$. Moreover, if $R$ is a local ring with unique maximal ideal $\MI$ (in which case we say ``the local ring $(R,\MI)$"), then a valuation $\nu$ on $R$ is said to be \textbf{centered} at $R$ if $\nu(a)\geq 0$ for every $a\in R$ and $\nu(a)>0$ for every $a\in \MI$. We observe that if $\nu$ is a valuation having a center on $R$, then $\nu$ is centered on $R_{\mathfrak C_\nu(R)}$. The \textbf{value group of $\nu$}, denoted by $\nu R$,  is defined as the subgroup of $\Gamma$ generated by $\{\nu(a)\mid a\in R\}$. The \textbf{rank of $\nu$} is the number of proper convex subgroups of $\nu R$.

For an element $b\in R\setminus \SU(\nu)$ we consider the canonical map $\Phi:R\lra R_b$ given by $\Phi(a)=a/1$. Let
\[
J(b):=\ker \Phi=\bigcup_{i=1}^\infty \ann_R(b^i).
\]
We have a natural embedding $R/J(b)\subseteq R_{b}$. Take $a_1,\ldots,a_r\in R$ such that
$$
\nu(a_i)\geq \nu(b)\quad\text{ for each }i,1\leq i\leq r.
$$
Consider the subring $R':=R/J(b)[a_1/b,\ldots,a_r/b]$ of $R_b$. Then the restriction of $\nu$ to $R'$ has a center $\mathfrak C_{\nu}(R')$ in $R'$. We set $R^{(1)}:=R'_{\mathfrak C_{\nu}(R')}$.

\begin{Def}\label{defblup}
The canonical map $R\lra R^{(1)}$ will be called the \textbf{local blowing up of $R$} with respect to $\nu$ along the ideal $(b,a_1,\cdots, a_r)$. For a valuation $\mu$ having a center on $R$ we will say that $R\lra R^{(1)}$ is \textbf{$\mu$-compatible} if $b\notin\mathfrak C_\mu(R)$ and $a_i\in \mathfrak C_\mu(R)$ for every $i$, $1\leq i\leq r$.
\end{Def}

\begin{Lema}\label{compobblupisblup}
The composition of finitely many local blowings up is again a local blowing up. Moreover, if each of these local blowings up is $\mu$-compatible, then their composition is again $\mu$-compatible.
\end{Lema}
\begin{proof}
It is enough to prove that for two local blowings up $\pi:R\lra R^{(1)}$ and $\pi':R^{(1)}\lra R^{(2)}$ with respect to $\nu$, there exists a local blowing up $R\lra R^{(3)}$ with respect to $\nu$ such that $R^{(3)}\simeq R^{(2)}$. We write
\[
R^{(1)}=R'_{\mathfrak C_\nu(R')}\mbox{ for }R'=R/J(b)[a_1/b,\ldots,a_r/b]
\]
for some $a_1,\ldots,a_r,b\in R$ and
\[
R^{(2)}=R'^{(1)}_{\mathfrak C_\nu(R'^{(1)})}\mbox{ for }R'^{(1)}=R^{(1)}/J(\beta)[\alpha_1/\beta,\ldots,\alpha_s/\beta]
\]
for some $\alpha_1,\ldots,\alpha_s,\beta\in R^{(1)}$. Then there exist $a_{r+1},\ldots,a_{r+s},b'\in R$ such that $\alpha_i/\beta=\pi(a_{r+i})/\pi(b')$ for each $i$, $1\leq i\leq s$. Consider the local blowing up
$$
R\lra R^{(3)}
$$
given by
\[
R^{(3)}=R''_{\mathfrak C_\nu(R'')}\mbox{ for }R''=R/J(bb')[a_1b'/bb',\ldots,a_rb'/bb',a_{r+1}b/bb'\ldots,a_{r+s}b/bb'].
\]
It is straightforward to prove that $R^{(2)}\simeq R^{(3)}$.
\end{proof}
In view of Lemma \ref{compobblupisblup}, we will freely use the fact that the composition of finitely many  local blowings up is itself a local blowing up without mentioning it explicitly.

For simplicity of notation, we denote the \textbf{nilradical} of $R$ by $I$, i.e.,
\[
I=\rad(R):=\{a\in R\mid a^l=0\mbox{ for some }l\in \N\}.
\]

\begin{Def}\label{defnormflat}
We say that $\Sp(R)$ is \textbf{normally flat along $\Sp(R_{\rm red})$} if $I^n/I^{n+1}$ is an $R_{\rm red}$-free module for every $n\in\N$.
\end{Def}

Since $R$ is noetherian, there exists $N\in\N$ such that $I^n=(0)$ for every $n>N$. Hence, the condition in Definition \ref{defnormflat} is equivalent to the freeness of the finitely many modules
$I/I^2,\ldots, I^N/I^{N+1}=I^N$.

\begin{Def}
For a local ring $R$, a valuation $\nu$ centered on $R$ is said to admit \textbf{local uniformization} if there exists a local blowing up $R\lra R^{(1)}$ with respect to $\nu$ such that $\left(R^{(1)}\right)_{\rm red}$ is regular and $\Sp\left(R^{(1)}\right)$ is normally flat along $\Sp\left(\left(R^{(1)}\right)_{\rm red}\right)$. 
\end{Def}

Let $\nu=\nu_1\circ\nu_2$ be a fixed decomposition of $\nu$. For simplicity of notation, we set $\PI:=\mathfrak C_{\nu_1}(R)$ and for a local blowing up $R\lra R^{(1)}$ we set $\PI^{(1)}:=\mathfrak C_{\nu_1}\left(R^{(1)}\right)$. We need to guarantee that the main structure of $R_\PI$ and $R/\PI$ are preserved under $\nu_1$-compatible local blowings up. More precisely, we have to prove the following:

\begin{Prop}\label{propabkeepreg}
Let $\pi:R\lra R^{(1)}$ be a $\nu_1$-compatible local blowing up. Then the canonical maps $R_\PI\lra R^{(1)}_{\PI^{(1)}}$ and $R/\PI\lra R^{(1)}/\PI^{(1)}$ induced by $\pi$ are isomorphisms.
\end{Prop}

In order to prove Proposition \ref{propabkeepreg} we need the following basic lemma.

\begin{Lema}\label{lemaabdouloc}
Let $S$ be a multiplicative system of $R$ contained in $R\setminus \mathfrak C_\nu(R)$. Then the canonical map
$\Phi:R_{\mathfrak C_\nu(R)}\lra \left(R_S\right)_{\mathfrak C_\nu(R_S)}$ given by $\Phi(a/b)=(a/1)/(b/1)$ is an isomorphism.
\end{Lema}

\begin{proof}
For an element $(a/b)/(c/d)\in \left(R_S\right)_{\mathfrak C_\nu(R_S)}$ we have
$$
\nu(b)=\nu(c)=\nu(d)=0.
$$
Consequently, $\nu(bc)=0$ and $ad/bc\in R_{\mathfrak C_\nu(R)}$. Then
\[
(a/b)/(c/d)=(ad/1)/(bc/1)=\Phi(ad/bc).
\]

Suppose that $\Phi(a/b)=0$. This means that there exists $c/d\in R_S\setminus \mathfrak C_\nu(R_S)$ such that $ac/d=0$ in $R_S$. Thus, there exists $s\in S$ such that $sac=0$. Moreover, since $c/d\notin\mathfrak C_\nu(R_S)$ we also have that $c\notin\mathfrak C_\nu(R)$. This and the fact that $s\in S\subseteq R\setminus\mathfrak C_\nu(R)$ imply that $sc\notin \mathfrak C_\nu(R)$. Hence, $a/b=0$ in $R_{\mathfrak C_\nu(R)}$ which is what we wanted to prove.
\end{proof}

\begin{proof}[Proof of Proposition \ref{propabkeepreg}]
Applying Lemma \ref{lemaabdouloc} to $R$ (with $S=\{1,b,b^2,\ldots\})$ and $R'$ (with $S'=R'\setminus \mathfrak C_\nu(R')$) and the valuation $\nu_1$, we obtain that the canonical maps $R_\PI\lra\left(R_b\right)_{\mathfrak C_{\nu_1}(R_b)}$ and $R'_{\mathfrak
C_{\nu_1}(R')}\lra R^{(1)}_{\PI^{(1)}}$, respectively, are isomorphisms. Hence, in order to prove the first assertion, it is enough to show that the canonical map $\left(R_b\right)_{\mathfrak C_{\nu_1}(R_b)}\longleftarrow R'_{\mathfrak C_{\nu_1}(R')}$ is an isomorphism.

Since $R'\subseteq R_b$ and $\mathfrak C_{\nu_1}(R')=R'\cap \mathfrak C_{\nu_1}(R_b)$ we have that $R'_{\mathfrak C_{\nu_1}(R')}\lra \left(R_b\right)_{\mathfrak C_{\nu_1}(R_b)}$ is injective. On the other hand, any element $(a/b^n)/(c/b^m)$ in $\left(R_b\right)_{\mathfrak C_{\nu_1}(R_b)}$ can be written as $(a b^m/1)/(c b^n/1)$ which is the image of $a b^m/c b^n$. Hence the map
$$
R'_{\mathfrak C_{\nu_1}(R')}\lra \left(R_b\right)_{\mathfrak C_{\nu_1}(R_b)}
$$
is surjective and consequently it is an isomorphism.

Set $R_0=R/J(b)$ and consider the induced map $R_0\lra R^{(1)}$. Since the canonical map $R\lra R_0$ is surjective, in order to prove the surjectivity of $R\lra R^{(1)}/\PI^{(1)}$, it is enough to show that $R_0\lra R^{(1)}/\PI^{(1)}$ is surjective. For an element $\alpha\in R^{(1)}$ we write $\alpha=p/q$ where $p=P(a_1/b,\ldots,a_r/b)$ and $q=Q(a_1/b,\ldots,a_r/b)$ for some
\[
P(X_1,\ldots,X_r),Q(X_1,\ldots,X_r)\in R_0[X_1,\ldots,X_r].
\]
Set $p_0=P(0,\ldots,0)$ and $q_0=Q(0,\ldots,0)$. Then
\[
p_1:=p-p_0=\sum_{i=1}^r a_i/b\cdot P_i(a_1/b,\cdots,a_r/b)
\]
and
\[
q_1:=q-q_0=\sum_{i=1}^r a_i/b\cdot Q_i(a_1/b,\cdots,a_r/b)
\]
for some $P_i,Q_i\in R_0[X_1,\ldots,X_r]$, $1\leq i\leq r$. Since $\nu_1(a_i/b)> 0$ we obtain that $\nu_1(p_1)>0$ and $\nu_1(q_1)>0$. This implies that
\begin{eqnarray}
\nu_1(q_0)&=&0,\\
\nu_1(q_0q)&=&0
\end{eqnarray}
and
\begin{equation}
\nu_1(q_0p_1-p_0q_1)>0.
\end{equation}
Therefore,
\[
p/q-p_0/q_0=(q_0p_1-p_0q_1)/q_0q\in \PI^{(1)}.
\]
It remains to prove that $p_0/q_0\in R_0$. Since $\nu_1(q_1)>0$, also $\nu(q_1)>0$. Hence, $\nu(q_0)=\nu(q-q_1)=0$ and consequently
$q_0$ is a unit in $R_0$. Therefore, $p_0/q_0\in R_0$.

To finish our proof it is enough to show that the kernel of $R\lra R^{(1)}/\PI^{(1)}$ is $\PI$. This follows immediately from the definition of $\PI$ and $\PI^{(1)}$ as the centers of $\nu_1$ on $R$ and $R^{(1)}$, respectively.
\end{proof}

Lemmas \ref{lifblupfrloc} and \ref{lifblupfrquo} below are generalizations of Lemma 2.18 and Corollary 2.20 of \cite{Nov1}, respectively. The proofs presented there can be adapted to our more general case. We present sketches of the proofs for the convenience of the reader.

\begin{Lema}\label{lifblupfrloc}
For each local blowing up $R_\PI\lra\widetilde R^{(1)}$ with respect to $\nu_1$, there exists a local blowing up $R\lra R^{(1)}$ with respect to $\nu$ such that $\widetilde R^{(1)}\simeq R^{(1)}_{\PI^{(1)}}$.
\end{Lema}
\begin{proof}
We consider the local blowing up $R_\PI\lra \widetilde{R}^{(1)}$ given by
\[
\widetilde R^{(1)}=\widetilde R'_{\mathfrak C_{\nu_1}(\widetilde R')}\mbox{ for }\widetilde R'=R_\PI/J(\beta)[\alpha_1/\beta,\ldots,\alpha_r/\beta].
\]
Choose $a_1,\ldots,a_r,b\in R$ such that for each $i$, $1\leq i\leq r$ we have $\Phi(a_i)/\Phi(b)=\alpha_i/\beta$ where $\Phi:R\lra
R_\PI$ is the canonical map. If $\nu(a_i)<\nu(b)$ for some $i$, $1\leq i\leq r$, then we have $\nu_1(\alpha_i)=\nu_1(\beta)$. Choose $i$ so as to minimize the value $\nu(a_i)$, in other words, so that $\nu(a_i)\le\nu(a_j)$ for all $j\in\{1,\dots,r\}$. Set
\[
\widetilde{R}'':=R_\PI/J(\alpha_i)\left[\frac{\alpha_1}{\alpha_i},\ldots,\frac{\alpha_{i-1}}{\alpha_i},\frac{\beta}{\alpha_i},
\frac{\alpha_{i+1}}{\alpha_i},\ldots,\frac{\alpha_r}{\alpha_i}\right].
\]
Then $R^{(1)}\simeq \widetilde{R}''_{\mathfrak C_{\nu_1}(\widetilde{R}'')}$. Hence, after a suitable permutation of the set
$\{a_1,\dots,a_r,b\}$, we may assume that $\nu(a_i)\geq \nu(b)$ for every $i$, $1\leq i\leq r$. Consider the local blowing up
\[
R^{(1)}=R'_{\mathfrak C_{\nu}(R')}\mbox{ for }R'=R/J(b)[a_1/b,\ldots,a_r/b]
\]
with respect to $\nu$. It is straightforward to prove that $R^{(1)}_{\PI^{(1)}}\simeq\widetilde R^{(1)}$.
\end{proof}

\begin{Lema}\label{lifblupfrquo}
For each local blowing up $R/\PI\lra \overline R^{(1)}$ with respect to $\nu_2$, there exists a local blowing up $R\lra R^{(1)}$ with respect to $\nu$ such that $R^{(1)}/\PI^{(1)}\simeq \overline R^{(1)}$ and $R_\PI\simeq R^{(1)}_{\PI^{(1)}}$.
\end{Lema}

\begin{proof}
For an element $a\in R$ we denote its image under the canonical map
$$
R\lra R/\PI
$$
by $\overline a$. Then
\[
\overline R^{(1)}=\overline R'_{\mathfrak C_{\nu_2}(\overline R')}\mbox{ with }\overline R'=\left(R/\PI\right)/J(\overline b)[\overline a_1/\overline b,\ldots,\overline a_r/\overline b]
\]
for some $a_1,\ldots,a_r,b\in R\setminus\mathfrak p$. Since $\nu_2(\overline a_i)\geq \nu_2(\overline b)$ we have $\nu(a_i)\geq\nu(b)$ for every $i$, $1\leq i\leq r$. Then we can consider the local blowing up
\[
R^{(1)}=R'_{\mathfrak C_{\nu}(R')}\mbox{ with }R'=R/J(b)[a_1/b,\ldots,a_r/b]
\]
with respect to $\nu$. It is again straightforward to prove that $R^{(1)}/\PI^{(1)}\simeq \overline R^{(1)}$ and $R_\PI\simeq
R^{(1)}_{\PI^{(1)}}$.
\end{proof}

\section{Associated prime ideals of $R$}\label{assocprimeideals}

Let $R$ be a local ring and $\nu$ a valuation centered on $R$. The main result of this section is the following:

\begin{Prop}\label{asspriide}
There exists a local blowing up $R\lra R^{(1)}$ with respect to $\nu$ such that $\rad\left(R^{(1)}\right)$ is the only associated prime of $R^{(1)}$.
\end{Prop}

In order to prove Proposition \ref{asspriide}, we need the following result.

\begin{Lema}\label{claimasspid}
Let $R'=R/J(b)[a_1/b,\ldots,a_r/b]\subseteq R_b$ for some $b,a_1,\ldots,a_r\in R$ with $\nu(b)\leq\nu(a_i)$ for every $i$, $1\leq i\leq r$. Then for every $c'\in R'$ the ideal $\ann_{R'}(c')$ can be written as $\ann_{R'}(c/1)$ for some $c\in R$. Moreover, if
$\ann_{R'}(c')$ is prime, then $\ann_R\left(b^Nc\right)$ is a prime ideal of $R$ for some $N\in\N$.
\end{Lema}

\begin{proof}
Choose $c\in R$ such that $c'=c/b^l$ for some $l\in \N$. Fix $a'\in R'$ and write $a'=a/b^m$ for some $m\in\N$ and $a\in R$. Then we have
\[
a'\in\ann_{R'}(c')\Llr acb^n=0\mbox{ for some }n\in\N\Llr a'\in \ann_{R'}(c/1).
\]

Now assume that $\ann_{R'}(c')$ is prime and set $R_0:=R/J(b)$. Then
\[
\ann_{R_0}(c/1)=\ann_{R'}(c')\cap R_0
\]
is also prime. Moreover,
\begin{equation}\label{inveqann}
\pi^{-1}(\ann_{R_0}(c/1))=\bigcup_{n=1}^\infty\ann_R(b^nc),
\end{equation}
where $\pi:R\lra R/J(b)$ is the canonical epimorphism. Indeed,
\begin{displaymath}
\begin{array}{rcl}
a\in \pi^{-1}(\ann_{R_0}(c/1))&\Llr & a c/1=0\mbox{ in }R_b\\
& \Llr & b^na c=0\mbox{ in }R\mbox{ for some }n\in \N\\
& \Llr & a\in \displaystyle\bigcup_{n=1}^\infty\ann_R(b^nc).
\end{array}
\end{displaymath}
Since $R$ is noetherian and
\[
\ann_R(bc)\subseteq \ann_R\left(b^2 c\right)\subseteq\cdots\subseteq \ann_R\left(b^n c\right)\subseteq\cdots
\]
we have that
\begin{equation}\label{anneqannofone}
\ann_R\left(b^Nc\right)=\displaystyle\bigcup_{n=1}^\infty\ann_R(b^nc)\mbox{ for some }N\in \N. 
\end{equation}
By (\ref{inveqann}) and (\ref{anneqannofone}) we conclude that $\ann_R\left(b^Nc\right)$ is a prime ideal of $R$.
\end{proof}

\begin{Cor}\label{nilrpresblup}
For a local blowing up $R\lra R^{(1)}$, if $\rad(R)$ is the only associated prime ideal of $R$, then $\rad\left(R^{(1)}\right)$ is the only associated prime ideal of $R^{(1)}$.
\end{Cor}
\begin{proof}
Let $R^{(1)}=R'_{\mathfrak C_\nu(R')}$ for some $R'$ as in Lemma \ref{claimasspid}. Theorem 6.2 of \cite{Mat} gives us that $\ass\left(R^{(1)}\right)=\ass(R')\cap \Sp\left(R^{(1)}\right)$. This and Lemma \ref{claimasspid} guarantee that $|\ass\left(R^{(1)}\right)|\leq |\ass\left(R\right)|=1$. Consequently, $R^{(1)}$ has only one associated prime ideal, say $\QI$. The primary decomposition theorem now gives us that $\QI=\rad\left(R^{(1)}\right)$, which is what we wanted to prove.
\end{proof}

We will use Corollary \ref{nilrpresblup} throughout this paper without always mentioning it explicitly.

\begin{proof}[Proof of Proposition \ref{asspriide}]
Since $\SU(\nu)$ is a minimal prime ideal, there exists at most one associated prime ideal of $R$ contained in (hence equal to)
$\SU(\nu)$. We will prove that if $|\ass(R)|>1$, then there exists a local blowing up $R\lra R^{(1)}$ such that $|\ass\left(R^{(1)}\right)|<|\ass(R)|$.

Take an associated prime ideal $\QI$ of $R$ such that $\QI\not\subseteq \SU(\nu)$. Write
$$
\QI=(b,a_1,\ldots,a_r)
$$
with
$\nu(b)\leq\nu(a_i)$ for every $i$, $1\leq i\leq r$. Blowing up $R$ with respect to $\nu$ along $\QI$ gives us a local ring
\[
R^{(1)}=R'_{\mathfrak C_\nu(R')}\mbox{ where }R'=R/J(b)[a_1/b,\ldots,a_r/b].
\]
Observe that this is indeed a local blowing up because $\nu(b)\leq \nu(a_i)$ for every $i$ and $\QI\not\subseteq \SU(\nu)$ implies that $b\notin \SU(\nu)$. Since
\[
\ass\left(R^{(1)}\right)=\ass(R')\cap \Sp\left(R^{(1)}\right)\mbox{ (see Theorem 6.2 of \cite{Mat}),}
\]
it remains to show that $|\ass(R')|<|\ass(R)|$.

By Lemma \ref{claimasspid}, we obtain that $R'$ has at most $|\ass(R)|$ many associated prime ideals. Moreover, for the chosen associated prime ideal $\QI=\ann_R(c)$ of $R$ and  for every $r\in\N$, the ideal $\ann_{R'}(c/b^r)$ is not prime in $R'$. Indeed, since $\QI=(b,a_1,\ldots,a_r)=\ann_R(c)$ we have $bc=0$ in $R$. This means that
$$
c/1=0/1
$$
in $R'$ and consequently $\ann_{R'}(c/1)=R'$ (which is not prime). Therefore
$$
|\ass(R')|<|\ass(R)|.
$$

\end{proof}

\begin{Obs}\label{rmkonasspide}
If $I$ is the only associated prime ideal of $R$, then for every $b\notin I$ we have $J(b)=(0)$. In this particular case, we can eliminate the ideal $J(b)$ in the definition of a local blowing up. We will use this throughout this paper without mentioning it explicitly.
\end{Obs}

\section{Making $R_{\rm red}$ regular}\label{sectiomakerpireg}
Let $R$ be a local ring and $\nu$ a valuation centered on $R$. Assume that $\nu=\nu_1\circ\nu_2$ and denote by $\PI$ the center of $\nu_1$ on $R$. As usual, we denote by $I$ the nilradical of $R$ and for a local blowing up $R\lra R^{(1)}$ we denote the nilradical of $R^{(1)}$ by $I_{(1)}$. Assume that $I$ is the only associated prime ideal of $R$. The main goal of this section is to prove the following proposition.

\begin{Prop}\label{propthatmakesreg}
Assume that $(R_\PI)_{\rm red}$ and $R/\PI$ are regular. Then there exists a $\nu_1$-compatible local blowing up $R\lra R^{(1)}$ such that $\left(R^{(1)}\right)_{\rm red}$ is regular. Moreover, for every local blowing up $R^{(1)}\lra R^{(2)}$ along an ideal $(b,a_1,\ldots,a_r)$ with $b\notin\PI^{(1)}$ and $a_1,\ldots,a_r\in I_{(1)}$ we have that $\left(R^{(2)}\right)_{\rm red}$ is regular.
\end{Prop}

In order to prove Proposition \ref{propthatmakesreg} we will need a few lemmas.

\begin{Lema}\label{propelimunfloc}
Assume that $(R_\PI)_{\rm red}$ is regular. Then there exists a $\nu_1$-compatible local blowing up $R\lra R^{(1)}$ such that the
$R^{(1)}/\PI^{(1)}$-module $\PI^{(1)}/\left(\left(\PI^{(1)}\right)^2+I_{(1)}\right)$ is free. Moreover, if
$y^{(1)}_1,\cdots,y^{(1)}_r$ are elements of $\PI^{(1)}$ whose images in\linebreak $\PI^{(1)}/\left(\left(\PI^{(1)}\right)^2+I_{(1)}\right)$ form a basis of $\PI^{(1)}/\left(\left(\PI^{(1)}\right)^2+I_{(1)}\right)$ then their images in
$\left(R^{(1)}_{\PI^{(1)}}\right)_{\rm red}$ form a regular system of parameters of $\left(R^{(1)}_{\PI^{(1)}}\right)_{\rm red}$.
\end{Lema}

\begin{Lema}\label{lemmathatkeepregul}
Let $\pi:R\lra R^{(1)}$ be a local blowing up along an ideal $(b,a_1,\ldots,a_r)$ with $b\notin\PI$ and $a_1,\ldots,a_r\in I$. If
$\PI/\left(\PI^2+I\right)$ is a free $R/\PI$-module, then
$$
\PI^{(1)}/\left(\left(\PI^{(1)}\right)^2+I_{(1)}\right)
$$
is a free $R^{(1)}/\PI^{(1)}$-module.
\end{Lema}

\begin{Lema}\label{proprelfreeandreg}
Take $y_1,\ldots,y_r\in \PI$ and $x_1,\ldots,x_t\in \MI\setminus\PI$ whose images form a regular system of parameters of
$\left(R_\PI\right)_{\rm red}$ and $R/\PI$, respectively. If $\PI/(\PI^2+I)$ is an $R/\PI$-free module with basis
$y_1+(\PI^2+I),\ldots,y_r+(\PI^2+I)$, then $R_{\rm red}$ is regular.
\end{Lema}

\begin{proof}[Proof of Proposition \ref{propthatmakesreg}, assuming Lemmas \ref{propelimunfloc}, \ref{lemmathatkeepregul} and
\ref{proprelfreeandreg}]
We apply Lemma \ref{propelimunfloc} to obtain a $\nu_1$-compatible local blowing up $R\lra R^{(1)}$ and $y^{(1)}_1,\ldots,y^{(1)}_r\in
R^{(1)}$ such that their images in $\PI^{(1)}/\left(\left(\PI^{(1)}\right)^2+I_{(1)}\right)$ form an $\left(R^{(1)}\right)_{\rm red}$- basis and their images in $\left(R^{(1)}_{\PI^{(1)}}\right)_{\rm red}$ form a regular system of parameters. Moreover, by Proposition \ref{propabkeepreg}, $R^{(1)}/\PI^{(1)}$ is regular. Also, by Lemma \ref{lemmathatkeepregul} and Proposition \ref{propabkeepreg}, for every local blowing up $R^{(1)}\lra R^{(2)}$ along an ideal $(b,a_1,\ldots,a_r)$ with $b\notin \PI^{(1)}$ and
$a_1,\ldots,a_r\in I_{(1)}$ the hypotheses of Lemma \ref{proprelfreeandreg} are satisfied for $R^{(2)}$. Hence, we obtain that
$\left(R^{(1)}\right)_{\rm red}$ and $\left(R^{(2)}\right)_{\rm red}$ are regular.
\end{proof}

We now proceed with the proofs of Lemmas \ref{propelimunfloc}, \ref{lemmathatkeepregul} and \ref{proprelfreeandreg}.
\begin{Lema}\label{presgenerunblup}
Take generators $y_1,\ldots,y_r,y_{r+1},\ldots,y_{r+s}$ of $\PI$ and $b\notin \PI$. Let
$$
\pi:R\lra R^{(1)}
$$
be the local blowing up along the ideal $(b,y_1,\ldots,y_r)$. Set
\[
y_i^{(1)}=\pi(y_i)/b\mbox{ for }1\leq i\leq r\mbox{ and }y_{r+k}^{(1)}=\pi(y_{r+k})\mbox{ for }1\leq k\leq s.
\]
Then $\PI^{(1)}$ is generated by $y_1^{(1)},\ldots,y_{r+s}^{(1)}$.
\end{Lema}
\begin{proof}
Obviously $y^{(1)}_i\in\PI^{(1)}$ for every $i$, $1\leq i\leq r+s$. Take an element $p/q\in \PI^{(1)}$. This implies that
$p=p(y_1/b,\ldots,y_r/b)$ for some $p(X_1,\ldots,X_r)\in R[X_1,\ldots,X_r]$ (see Remark \ref{rmkonasspide}). If we set $p_0=p(0,\ldots,0)$, then
\[
p=p_0+\frac{y_1}{b}p_1+\ldots+\frac{y_r}{b}p_r\mbox{, for some }p_1,\ldots,p_r\in R'.
\]
This implies that $p_0\in \PI$. Hence, there exist
$a_1,\ldots,a_{r+s}\in R$ such that $p_0=a_1y_1+\ldots+a_{r+s}y_{r+s}$. Thus
\[
\frac pq=\sum_{i=1}^r\frac{\pi(ba_i)+p_i}{q}y_i^{(1)}+\sum_{k=1}^{s}\frac{\pi(a_{r+k})}{q}y^{(1)}_{r+k}\in \left(y_1^{(1)},\ldots,y_{r+s}^{(1)}\right)R^{(1)}.
\]
This concludes our proof.
\end{proof}

\begin{proof}[Proof of Lemma \ref{propelimunfloc}] Since $(R_\PI)_{\rm red}$ is regular there are elements $y_1,\ldots,y_r\in\PI$ such that their images in $(R_\PI)_{\rm red}$ form a regular system of parameters. The first step is to reduce to the case when
$y_1,\ldots,y_r$ generate $\PI$.

Assume that $y_1,\ldots,y_r$ do not generate $\PI$. Choose $y_{r+1},\ldots,y_{r+s}\in\PI$ such that
$y_1,\ldots,y_r,y_{r+1},\ldots,y_{r+s}$ generate $\PI$. For each $k$, $1\leq k\leq s$, we can find $b_k\in R\setminus\PI$, $b_{1k},\ldots, b_{rk}\in R$ and $h_k\in (y_1,\ldots,y_r)^2$ such that
\[
b_ky_{r+k}+b_{1k}y_1+\ldots+b_{rk}y_r+h_k\in I.
\]
Consider the local blowing up $\pi:R\lra R^{(1)}$ along $(b_1,y_1,\ldots,y_r)$. It follows that
\[
\pi(b_1)\left(y^{(1)}_{r+1}+\pi(b_{11})y_1^{(1)}+\ldots+\pi(b_{r1})y^{(1)}_r+h_1^{(1)}\right)\in I_{(1)}
\]
and
\[
\pi(b_k)y^{(1)}_{r+k}+\pi(b_1b_{1k})y_1^{(1)}+\ldots+\pi(b_1b_{rk})y^{(1)}_r+h_k^{(1)}\in I_{(1)}\mbox{ for }2\leq k\leq s,
\]
where $y_i^{(1)}=\pi(y_i)/{b_1}$ for $1\leq i\leq r$ and $y_{r+k}^{(1)}=\pi(y_{r+k})$ and some $h^{(1)}_k\in \left(y_1^{(1)},\ldots,y_r^{(1)}\right)^2$ for $1\leq i\leq s$. Since $I_{(1)}$ is prime and $\pi(b_1)\notin I_{(1)}$ we obtain that 
\[
y^{(1)}_{r+1}+\pi(b_{11})y_1^{(1)}+\ldots+\pi(b_{r1})y^{(1)}_r+h_1^{(1)}\in I_{(1)}.
\]
Consequently,
\[
\left(y^{(1)}_1,\ldots,y^{(1)}_r,y^{(1)}_{r+1},\ldots,y^{(1)}_{r+s}\right)R^{(1)}=\left(y^{(1)}_1,\ldots,y^{(1)}_r,y^{(1)}_{r+2},\ldots,y^{(1)}_{r+s}\right)R^{(1)}.
\]
We proceed inductively to obtain a $\nu_1$-compatible local blowing up $R\lra R^{(s)}$ such that
\[
\left(y^{(s)}_1,\ldots,y^{(s)}_r,y^{(s)}_{r+1},\ldots,y^{(s)}_{r+s}\right)R^{(s)}=\left(y^{(s)}_1,\ldots,y^{(s)}_r\right)R^{(s)}.
\]
By Lemma \ref{presgenerunblup}, we have $\PI^{(s)}=\left(y^{(s)}_1,\ldots,y^{(s)}_r,y^{(s)}_{r+1},\ldots,y^{(s)}_{r+s}\right)R^{(s)}$ and by Lemma \ref{propabkeepreg} the images of $y^{(s)}_1,\ldots,y^{(s)}_r$ in $\left(R^{(s)}_{\PI^{(s)}}\right)_{\rm red}$ form a regular system of parameters. This means that $y^{(s)}_1,\ldots,y^{(s)}_r$ generate $\PI^{(s)}$. Thus we have reduced the problem to the case when $\left(y_1,\ldots,y_r\right)$ generate $\PI$ and will make this assumption from now on.

Now, the only non-trivial fact that remains to be checked is that the images of $y_1,\ldots,y_r$ in $\PI/\PI^2+I$ are $R/\PI$-linearly independent. Take $a_1,\ldots,a_r\in R$ such that
\[
a_1y_1+\ldots+a_ry_r\in \PI^2+I.
\]
Since the images of $y_1,\ldots,y_r$ in $\left(R_\PI\right)_{\rm red}$ form a regular system of parameters, their images in $\PI
R_\PI/(\PI^2+I)R_\PI$ form an $R_\PI/\PI R_\PI$-basis of $\PI R_\PI/(\PI^2+I)R_\PI$. This implies that $a_1/1,\ldots,a_r/1\in\PI R_\PI$ and consequently $a_1,\ldots,a_r\in\PI$.

This completes the proof of the Lemma.
\end{proof}

\begin{proof}[Proof of Lemma \ref{lemmathatkeepregul}]
Take $y_1,\ldots,y_s\in\PI$ such that their images form an $R/\PI$-basis of $\PI/\left(\PI^2+I\right)$. We claim that the images of $\pi(y_1),\ldots,\pi(y_s)$ form an $R^{(1)}/\PI^{(1)}$-basis of $\PI^{(1)}/\left(\left(\PI^{(1)}\right)^2+I_{(1)}\right)$. Take an element $\alpha\in \PI^{(1)}$. Then $\alpha=p/q$ where $p,q\in R':=R[a_1/b,\ldots,a_r/b]$ with $\nu_1(p)>0$ and $\nu(q)=0$. Set $p_0=p(0,\ldots,0)$ and write
\[
p=p_0+\frac{a_1}{b}p_1+\ldots+\frac{a_r}{b}p_r\mbox{ for some }p_1,\ldots,p_r\in R'.
\]
This implies that $p_0\in\PI$. By our assumption, there exist $c_1,\ldots,c_s\in\PI$, $g\in\PI^2$ and $h\in I$ such that
\[
p_0=c_1y_1+\ldots+c_sy_r+g+h.
\]
Consequently,
\[
\alpha=\frac{\pi(c_1)}{q}\pi(y_1)+\ldots+\frac{\pi(c_s)}{q}\pi(y_s)+\frac{\pi(g)}{q}+\frac{\pi(h)}{q}+\frac{a_1}{b}\frac{p_1}{q}+\ldots+\frac{a_r}{b}\frac{p_r}{q}.
\]
Since $a_1,\ldots,a_r,h\in I$, we have that
\[
\frac{\pi(h)}{q}+\frac{a_1}{b}\frac{p_1}{q}+\ldots+\frac{a_r}{b}\frac{p_r}{q}\in I_{(1)}.
\]
This and the fact that $\pi(g)/q\in \left(\PI^{(1)}\right)^2$ imply that the images of $\pi(y_1),\ldots,\pi(y_r)$ generate
$\PI^{(1)}/\left(\left(\PI^{(1)}\right)^2+I_{(1)}\right)$.

Now assume that there exists $\alpha_i=(a_i/b^{l_i})/(c_i/b^{m_i})\in R^{(1)}$, $1\leq i\leq r$, such that
\[
\alpha_1\pi(y_1)+\ldots+\alpha_r\pi(y_r)\in \left(\PI^{(1)}\right)^2+I_{(1)}.
\]
Then there exists $n\in \N$ such that
\[
a_1b^ny_1+\ldots+a_rb^ny_r\in \PI^2+I.
\]
This implies that $a_ib^n\in\PI$ for every $i$, $1\leq i\leq r$. Since $b\notin\PI$, this implies that $a_1,\ldots,a_r\in\PI$. Therefore, $\alpha_1,\ldots,\alpha_r\in \PI^{(1)}$, which concludes our proof.
\end{proof}

\begin{proof}[Proof of Lemma \ref{proprelfreeandreg}]
Set $\PI'=\{a+I\in R_{\rm red}\mid a\in\PI\}$. Since the images of the $y_i$'s in $\PI/(\PI^2+I)$ form a basis of $\PI/(\PI^2+I)$, we conclude that $(y_1,\ldots,y_r)+\PI^2+I=\PI$. Applying Nakayama's Lemma (corollary of Theorem 2.2 of \cite{Mat}) we conclude that
$(y_1,\ldots,y_r)+I=\PI$ and consequently $y_1+I,\ldots, y_r+I$ generate $\PI'$.

Since the images of $y_1,\ldots,y_r,x_1,\ldots,x_t$ in $R_{\rm red}$ generate $\MI'=\{a+I\in R_{\rm red}\mid a\in\MI\}$ we conclude that $r+t\geq \dim R_{\rm red}$. Also, since $r= \dim \left(R_\mathfrak{p}\right)_{\rm red}=\hei
\left(\PI'\right)$ and $t=\dim \left(R/\PI\right)=\hei(\MI/\PI)=\hei\left(\MI'/\PI'\right)$ we have
\[
\dim (R_{\rm red})=\hei\left(\MI'\right)\geq \hei\left(\PI'\right)+\hei\left(\MI'/\PI'\right)=r+t\geq \dim (R_{\rm red}).
\]
Therefore, $r+t=\dim (R_{\rm red})$ and hence $R_{\rm red}$ is regular.
\end{proof}

\section{Making $I^n/I^{n+1}$ free}\label{sectiomakeinin1free}

Let $R$ be a local ring and $\nu$ a valuation centered on $R$. Assume that
$$
\nu=\nu_1\circ\nu_2
$$
and denote by $\PI$ the center of $\nu_1$ on $R$. As usual, we set $I=\rad(R)$ and $I_\PI:=\rad(R_\PI)$. Also, for a local blowing up $R\lra R^{(k)}$ we set $I_{(k)}=\rad\left(R^{(k)}\right)$ and $I_{\PI^{(k)}}:=\rad\left(R^{(k)}_{\PI^{(k)}}\right)$. Assume that $I$ is the only associated prime ideal of $R$. The main goal of this section is to prove the following proposition.

\begin{Prop}\label{makmodfree}
Assume that $I_\PI^n/I_\PI^{n+1}$ is an $\left(R_\PI\right)_{\rm red}$-free module for every $n\in\N$. Then there exists a local blowing up $R\lra R^{(1)}$ with respect to $\nu$ along an ideal $(b,a_1,\ldots,a_r)$ with $b\notin\PI$ and $a_1,\ldots,a_r\in I$ such that the $\left(R^{(1)}\right)_{\rm red}$-module $I_{(1)}^n/I_{(1)}^{n+1}$ is free for every $n\in\N$.
\end{Prop}
In order to prove Proposition \ref{makmodfree}, we will need some preliminary results.
\begin{Lema}\label{Lemabgenofimodinq}
Take elements $y_1,\ldots,y_{r+s}\in I^n$ such that their images in $I^n/I^{n+1}$ generate $I^n/I^{n+1}$ as an $R_{\rm red}$-module. Consider the local blowing up $\pi:R\lra R^{(1)}$ along the ideal $(b,y_1,\ldots,y_r)$ for some $b\in R\setminus I$. Set
\[
y_i^{(1)}=\pi(y_i)/b\mbox{ for }1\leq i\leq r\mbox{ and }y_{r+k}^{(1)}=\pi(y_{r+k})\mbox{ for }1\leq k\leq s.
\]
Then the images of $y_1^{(1)},\ldots,y_{r+s}^{(1)}$ in $I_{(1)}^n/I_{(1)}^{n+1}$ form a set of generators of this module.
\end{Lema}
\begin{proof}
Take an element $p/q\in I_{(1)}^n$. As in proof of the Lemma \ref{presgenerunblup}, we can write
\[
p=p_0+\frac{y_1}{b}p_1+\ldots+\frac{y_r}{b}p_r\mbox{, for some }p_1,\ldots,p_r\in R'
\]
with $p_0\in I^n$. This means that there exists $a_1,\ldots,a_{r+s}\in R$ such that
$p_0-a_1y_1-\ldots-a_{r+s}y_{r+s}=y_0\in I^{n+1}$. Consequently,
\[
\frac pq-\sum_{i=1}^r\frac{\pi(ba_i)+p_i}{q}y_i^{(1)}-\sum_{i=r+1}^{r+s}\frac{\pi(a_{i})}{q}y^{(1)}_{i}=\frac{\pi(y_0)}{q}\in
I_{(1)}^{n+1}.
\]
This concludes our proof.
\end{proof}

\begin{Lema}\label{Lemathatsaysaboyutlineinde}
Under the same assumptions as in the previous lemma, if the images of $y_1,\ldots,y_r$ in $I^n/I^{n+1}$ are $R_{\rm red}$-linearly independent, then the images of $a^{(1)}_1,\ldots,a^{(1)}_r$ in $I_{(1)}^n/I_{(1)}^{n+1}$ are $\left(R^{(1)}\right)_{\rm red}$-linearly independent.
\end{Lema}
\begin{proof}

Take elements $\alpha_1,\ldots,\alpha_r\in R^{(1)}$ such that
\begin{equation}\label{deprelinradblup}
\alpha_1 y^{(1)}+\ldots+\alpha_r y^{(1)}\in I_{(1)}^{n+1}.
\end{equation}
We have to show that $\alpha_1,\ldots,\alpha_r\in I_{(1)}$. For each $i$, $1\leq i\leq r$, we write
$\alpha_i=(a_i/b^{r_i})/(c_i/b^{s_i})$ for some $a_i,c_i\in R$ and $r_i,s_i\in\N$. Then equation (\ref{deprelinradblup}) implies that there exists $l\in\N$ and $c\in R\setminus \PI$ such that
\[
a_1b^lcy_1+\ldots+a_rb^lcy_r\in I^{n+1}.
\]
Since $y_1+I^{n+1},\ldots, y_r+I^{n+1}$ are $R_{\rm red}$-linearly independent, this implies that
$$
a_ib^lc\in I\quad\text{ for every }i,1\leq i\leq r.
$$
Since $I$ is prime (this is a consequence of the fact that it is the only associated prime ideal of $R$) and $b,c\in R\setminus I$, we obtain that $a_1,\ldots,a_r\in I$. Consequently, $\alpha_1,\ldots,\alpha_r\in I_{(1)}$, which concludes our proof.
\end{proof}

\begin{proof}[Proof of Proposition \ref{makmodfree}]
By assumption, we have that $I_\PI^n/I_\PI^{n+1}$ is $\left(R_\PI\right)_{\rm red}$-free for every $n\in\N$. Hence, by Proposition
\ref{propabkeepreg} for every $\nu_1$-compatible local blowing up $R\lra R^{(1)}$ we have that $I_{\PI^{(1)}}^n/I_{\PI^{(1)}}^{n+1}$ is $\left(R^{(1)}_{\PI^{(1)}}\right)_{\rm red}$-free for every $n\in\N$. Therefore, it is enough to show that for a fixed $n\in\N$, there exists a local blowing up $R\lra R^{(1)}$ along an ideal $(b,a_1,\ldots,a_r)$ with $b\notin\PI$ and $a_1,\ldots,a_r\in I$ such that
$I_{(1)}^n/I_{(1)}^{n+1}$ is $\left(R^{(1)}\right)_{\rm red}$-free.

Take elements $y_1/b_1,\ldots, y_r/b_r\in I_\PI^n$, $y_1,\ldots,y_r\in R$ and $b_1,\ldots,b_r\in R\setminus \PI$ such that
\[
y_1/b_1+I_\PI^{n+1},\ldots,y_r/b_r+I_\PI^{n+1}
\]
form a basis of $I_\PI^n/I_\PI^{n+1}$. We observe first that since $I$ is prime and $y_i/b_i\in I_\PI^n$, we have $y_i\in I^n$ for each $i$, $1\leq i\leq r$. We claim that if
\[
y_1+I^{n+1},\ldots,y_r+I^{n+1}
\]
generate $I^n/I^{n+1}$ as an $R_{\rm red}$-module, then this module is free. Indeed, if there exists $a_i+I\in R_{\rm red}$ such that $a_1y_1+\ldots+a_ry_r\in I^{n+1}$, then
\[
a_1b_1/1\cdot y_1/b_1+\ldots+a_rb_r/1\cdot y_r/b_r=(a_1y_1+\ldots+a_ry_r)/1\in I_\PI^{n+1}.
\]
This implies that for each $i$, $1\leq i\leq r$, $a_ib_i/1\in I_\PI$ and consequently $a_ib_ic_i\in I$ for some 
$c_i\in R\setminus\PI$ . Since $I$ is prime and $b_1c_1,\ldots,b_rc_r\in R\setminus I$, we conclude that $a_1,\ldots,a_r\in I$, which is what we wanted to prove.

If $y_1+I^{n+1},\ldots,y_r+I^{n+1}$ do not generate $I^n/I^{n+1}$ (as an $R_{\rm red}$-module), then we take $y_{r+1},\ldots,y_{r+s}\in I^n$ such that $y_1+I^{n+1},\ldots,y_{r+s}+I^{n+1}$ generate $I^n/I^{n+1}$. For each $k$, $1\leq k\leq s$, since $y_{r+k}\in I^n$ there exist $b_k\in R\setminus \PI$, such that
\begin{equation}\label{eqthatrelatesgen}
b_ky_{r+k}-b_{1k}y_1-\ldots-b_{rk}y_r\in I^{n+1},
\end{equation}
for some $b_{1k},\ldots,b_{rk}\in R$. Consider now the local blowing up along the ideal $(b_1,y_1,\ldots,y_r)$. Set
$$
y^{(1)}_i:=\pi(y_i)/b_1\in R^{(1)}\quad\text{ for each }i,1\leq i\leq r
$$
and
$$
y^{(1)}_{r+k}:=\pi(y_{r+k})\in R^{(1)}\quad\text{ for each }k,1\leq k\leq s.
$$
From equation (\ref{eqthatrelatesgen}) we obtain that
\[
y^{(1)}_{r+1}-\pi(b_{11})y_1^{(1)}-\ldots-\pi(b_{r1})y_r^{(1)}\in I_{(1)}^{n+1},
\]
and
\[
\pi(b_k)y^{(1)}_{r+k}-\pi(b_1b_{1k})y_1^{(1)}-\ldots-\pi(b_rb_{rk})y_r^{(1)}\in I_{(1)}^{n+1},
\]
for every $k$, $2\leq k\leq s$. Consequently, $y^{(1)}_{r+1}+I_{(1)}^{n+1}$ is generated in the $\left(R^{(1)}\right)_{\rm red}$-module $I_{(1)}^n/I_{(1)}^{n+1}$ by $y^{(1)}_1+I_{(1)}^{n+1},\ldots, y_r^{(1)}+I_{(1)}^{n+1}$. Moreover, using Lemma \ref{Lemabgenofimodinq}, we obtain that $I_{(1)}^n/I_{(1)}^{n+1}$ is generated as an $R^{(1)}_{\rm red}$-module by the images of
$$
y_1^{(1)},\ldots,y^{(1)}_r,y^{(1)}_{r+2},\ldots,y^{(1)}_{r+s}.
$$
Also, by Lemma \ref{Lemathatsaysaboyutlineinde}, the images of $y_1^{(1)},\ldots,y_r^{(1)}$ in $I_{(1)}^n/I_{(1)}^{n+1}$ are
$\left(R^{(1)}\right)_{\rm red}$-linearly independent.

We proceed inductively to obtain a local blowing up $R\lra R^{(s)}$ such that the $\left(R^{(s)}\right)_{\rm red}$-module
$I^n_{(s)}/I_{(s)}^{n+1}$ is generated by the images of $y_1^{(s)},\ldots,y^{(s)}_r$ and the images of $y_1^{(s)},\ldots,y_r^{(s)}$ in
$I_{(s)}^n/I_{(s)}^{n+1}$ are $\left(R^{(s)}\right)_{\rm red}$-linearly independent.
\end{proof}

\section{Proof of the main Theorem}

In this section we present the proof of our main theorem.

\begin{proof}[Proof of Theorem \ref{mainthm}]
We will prove the assertion by induction on the rank. Since all rank one valuations admit local uniformization by assumption, we fix
$n\in\N$ and will prove that if all valuations of rank smaller than $n$ admit local uniformization, then also valuations of rank $n$ admit local uniformization.

Let $\nu$ be a valuation centered in the local ring $R\in Ob(\mathcal M)$ such that $\rk(\nu)=n$. By Lemma \ref{asspriide}, there exists a local blowing up $R\lra R^{(1)}$ with respect to $\nu$ such that $\rad\left(R^{(1)}\right)$ is the only associated prime ideal of $R^{(1)}$. Hence, replacing $R$ by $R^{(1)}$, we may assume that the only associated prime ideal of $R$ is $\rad(R)$.

Decompose $\nu$ as $\nu=\nu_1\circ\nu_2$ for valuations $\nu_1$ and $\nu_2$ with rank smaller than $n$. By assumption, we know that
$\nu_1$ and $\nu_2$ admit local uniformization. Since $\nu_1$ admits local uniformization, by use of Lemma \ref{lifblupfrloc}, there exists a local blowing up $R\lra R^{(1)}$ with respect to $\nu$ such that $R^{(1)}_{\PI^{(1)}}$ is regular and
$I_{\PI^{(1)}}^n/I^{n+1}_{\PI^{(1)}}$ is $\left(R^{(1)}_{\PI^{(1)}}\right)_{\rm red}$-free for every $n\in \N$. Replacing $R$ by
$R^{(1)}$ we may assume that $(R_\PI)_{\rm red}$ is regular and $I^n_\PI/I^{n+1}_\PI$ is $(R_\PI)_{\rm red}$-free for every $n\in \N$.

Since $\nu_2$ admits local uniformization, we can use Lemma \ref{lifblupfrquo} to obtain that there exists a local blowing up $R\lra R^{(1)}$ with respect to $\nu$ such that $\left(R^{(1)}_{\PI^{(1)}}\right)_{\rm red}$ and $R^{(1)}/\PI^{(1)}$ are regular and $I^n_{\PI^{(1)}}/I^{n+1}_{\PI^{(1)}}$ is $\left(R^{(1)}_{\PI^{(1)}}\right)_{\rm red}$-free for every $n\in \N$. Replacing $R$ by $R^{(1)}$ we can assume that $(R_\PI)_{\rm red}$ and $R/\PI$ are regular and that $I^n_\PI/I^{n+1}_\PI$ is $(R_\PI)_{\rm red}$-free for every $n\in \N$.

Since $\left(R_\PI\right)_{\rm red}$ and $R/\PI$ are regular, we apply Proposition \ref{propthatmakesreg} to obtain a $\nu_1$-compatible local blowing up $R\lra R^{(1)}$ such that $\left(R^{(1)}\right)_{\rm red}$ is regular. Using Proposition \ref{propabkeepreg}, we have that $I_{\PI^{(1)}}^n/I_{\PI^{(1)}}^{n+1}$ is a free $\left(R^{(1)}_{\PI^{(1)}}\right)_{\rm red}$-module for every $n\in\N$. By Proposition \ref{makmodfree}, there exists a $\PI^{(1)}$-compatible local blowing up $R^{(1)}\lra R^{(2)}$ such that $I_{(2)}^n/I_{(2)}^{n+1}$ is an $\left(R^{(2)}\right)_{\rm red}$-free module for every $n\in \N$. Moreover, since this local blowing up is along an ideal $(b,a_1,\ldots,a_r)$ with $b\notin\PI^{(1)}$ and $a_1,\ldots,a_r\in I_{(1)}$, we conclude using Proposition \ref{propthatmakesreg} that $\left(R^{(2)}\right)_{\rm red}$ is regular. This concludes our proof.
\end{proof}

\end{document}